\documentclass[11pt]{article}

\usepackage{amsmath,amssymb,amsthm,amsfonts}
\usepackage{graphicx}
\usepackage[round]{natbib}
\usepackage[hidelinks]{hyperref}
\usepackage{setspace}
\usepackage{array}
\usepackage{booktabs}
\usepackage[dvipsnames]{xcolor}

\renewcommand{\arraystretch}{1.3}

\theoremstyle{plain}
\newtheorem{theorem}{Theorem}[section]
\newtheorem{proposition}[theorem]{Proposition}
\newtheorem{lemma}[theorem]{Lemma}

\theoremstyle{definition}
\newtheorem{definition}[theorem]{Definition}
\newtheorem{assumption}[theorem]{Assumption}
\newtheorem{remark}[theorem]{Remark}

\DeclareMathOperator{\Var}{Var}

\newcommand{\E}{\mathbb{E}}
\newcommand{\R}{\mathbb{R}}
\newcommand{\N}{\mathcal{N}}
\newcommand{\F}{\mathcal{F}}

\title{Action-Space Entropy Regularization in Bayesian Markowitz}

\author{
Andy Au\\
Department of Mathematics and Statistics\\
Boston University\\
\texttt{aa314@bu.edu}\\[6pt]
}

\date{\today}

\begin{document}
\maketitle


\begin{abstract}
We solve the entropy-regularized mean--variance portfolio problem under Bayesian drift uncertainty. We combine continuous-time Bayesian filtering with stochastic policy optimization; the main finding is negative: the two mechanisms are orthogonal. Posterior dynamics are policy-independent, so entropy regularization cannot accelerate learning about the unknown drift. The mean control is identical to the deterministic Bayesian Markowitz feedback, and entropy enters only through policy variance. On the technical side, the optimal policy is Gaussian, the value function is quadratic in wealth, and the leading belief-dependent coefficient closes in exponential form. The framework recovers both parent models as limiting cases.
\end{abstract}

\newpage
\section{Introduction}
\label{sec:intro}

Portfolio optimization under parameter uncertainty is a fundamental problem in finance. The classical mean--variance framework of \citet{markowitz1952} produces clean solutions when parameters are known, but real markets have unknown drift. The difficulty of estimating expected returns \citep{merton1980} leads to unstable allocations when one naively plugs in estimated means.

Two recent lines of work address this from different angles.

\begin{enumerate}
     \item \textbf{Bayesian learning:} \citet{defranco2018} solve the continuous-time Markowitz problem when drift is unknown. Using Kalman--Bucy filtering, they reduce the problem to a semilinear PDE in the belief state and obtain closed-form solutions for Gaussian priors. The optimal control is deterministic feedback.
    
    \item \textbf{Entropy-regularized control:} \citet{wang2019} study mean--variance optimization with known drift but replace deterministic controls with stochastic policies. Adding an entropy bonus yields tractable Gaussian solutions and connects to relaxed stochastic control. The problem closes through the associated HJB equation.
\end{enumerate}

This paper considers the intersection: entropy-regularized control when drift is unknown. A natural hope is that the interaction produces behavior absent from either model individually---that entropy regularization might serve as a form of exploration under drift uncertainty. At the level of optimal behavior it does not: the two mechanisms turn out to be orthogonal. The one genuine interaction effect is technical. The closed-form optimizer is admissible only under the parameter restriction $P_0 T < 7/2$ (Proposition~\ref{prop:admissibility}), a threshold absent from both parent models.

\subsection{Summary of Results}

\begin{enumerate}
    \item \textbf{Closed-form solution.} Under partial information, the optimal policy is Gaussian and the value function is quadratic in wealth with coefficients $A(t,m) = e^{\alpha(t)m^2 + \gamma(t)}$ in closed form (Lemma~\ref{lem:gaussian}, Proposition~\ref{prop:closed-form}). The framework recovers \citet{wang2019} as $P_0 \to 0$ and \citet{defranco2018} as $\tau \to 0$.
    
    \item \textbf{Orthogonality.} Posterior dynamics are policy-independent, so entropy regularization cannot accelerate learning about $\rho$. The mean control is $\tau$-independent and identical to the deterministic Bayesian Markowitz feedback; entropy enters only the policy variance.
    
    \item \textbf{Belief-dependent policy variance.} The optimal policy variance grows exponentially in $m_t^2$ along sample paths: $\varsigma^{*2} = (\tau/2\sigma^2)\exp(|\alpha(t)|m_t^2 - \gamma(t))$. This is the analogue of the known-drift result $\varsigma^{*2} \propto \exp(\rho^2(T-t))$ from \citet{wang2019}, with $m_t$ replacing $\rho$ and $\alpha(t)$ absorbing $P_t$ (Proposition~\ref{prop:beliefvar}). The quantitative form is different but the qualitative story is the same: variance grows with the magnitude of the posterior mean driving the mean control.

    \item \textbf{Admissibility threshold.} The closed-form optimizer is admissible if and only if $P_0 T < 7/2$ (Proposition~\ref{prop:admissibility}). The optimal policy variance grows exponentially in $m_t^2$, and this restriction is exactly what keeps its second moment integrable against the Gaussian fluctuations of $m_t$. Neither parent model exhibits such a constraint.
\end{enumerate}

\begin{center}
\begin{tabular}{lcc}
\toprule
 & Known drift ($P_0=0$) & Unknown drift ($P_0>0$) \\
\midrule
Deterministic ($\tau=0$) & Classical MV & \citet{defranco2018} \\
Entropy-regularized ($\tau>0$) & \citet{wang2019} & \textbf{This paper} \\
\bottomrule
\end{tabular}
\end{center}

\subsection{Key Notation}

The most frequently used symbols; a complete table appears in Appendix~\ref{app:notation}.

\begin{center}
\renewcommand{\arraystretch}{1.35}
\begin{tabular}{p{1.8cm} p{11cm}}
\toprule
$\rho$ & Sharpe ratio $(\mu - r)/\sigma$ (unknown, to be learned) \\
$m_t, P_t$ & Posterior mean and variance of $\rho$ given observations to time $t$ \\
$X_t$ & Discounted wealth \\
$\pi_t$ & Policy: probability distribution over positions at time $t$ \\
$\tau$ & Entropy weight (regularization strength) \\
$A(t,m)$ & Leading coefficient in value function ansatz; $V_{xx} = 2A$ \\
\bottomrule
\end{tabular}
\end{center}

\section{Problem Formulation}
\label{sec:setup}

\subsection{Market Model}

Consider a market with a risk-free asset earning rate $r$ and a risky asset with price $S_t$ following
\[
    \frac{dS_t}{S_t} = \mu\,dt + \sigma\,dW_t,
\]
where $\sigma > 0$ is known and $\mu$ is unknown. We parameterize the unknown drift by the Sharpe ratio $\rho := \frac{\mu - r}{\sigma}$.

\begin{assumption}
\label{assump:prior}
The Sharpe ratio $\rho$ is independent of $W$ and has Gaussian prior: $\rho \sim \N(m_0, P_0)$.
\end{assumption}

We model the excess drift $\mu - r$ as constant but unknown under the physical measure. The investor's uncertainty is captured by the prior on $\rho$.

\subsection{Bayesian Filtering}

The investor observes prices but not $\rho$ directly. Define the \emph{whitened} cumulative excess return
\[
    Y_t := \int_0^t \frac{1}{\sigma}\left(\frac{dS_s}{S_s} - r\,ds\right),
\]
which satisfies $dY_t = \rho\,dt + dW_t$. Let $\F_t := \sigma(Y_s : s \le t)$ be the observation filtration---equivalently $\sigma(S_s : s \le t)$, since $\sigma$ and $r$ are known---and define the posterior moments
\[
    m_t := \E[\rho \mid \F_t], \qquad P_t := \Var(\rho \mid \F_t).
\]

\begin{proposition}[Posterior dynamics]
\label{prop:filter}
The posterior remains Gaussian: $\rho \mid \F_t \sim \N(m_t, P_t)$, with dynamics
\begin{align}
    dm_t &= P_t\,d\widehat{W}_t, \label{eq:dm} \\
    dP_t &= -P_t^2\,dt, \label{eq:dP}
\end{align}
where the innovation $d\widehat{W}_t := dY_t - m_t\,dt$ is an $(\F_t)$-Brownian motion. The posterior variance has the closed form $P_t = \frac{P_0}{1 + P_0 t}$.
\end{proposition}

This follows from standard Kalman--Bucy filtering; see, e.g., \citet{defranco2018}. The closed form $P_t = P_0/(1+P_0 t)$ is deterministic and $m_t$ is a martingale. The structural fact we use repeatedly is that posterior dynamics depend on $Y$ rather than the control $\pi$: \emph{filtering is policy-independent}.

\subsection{Wealth Dynamics and Bilinear Structure}

Let $X_t := e^{-rt}\bar{X}_t$ be discounted wealth and $u_t$ the discounted dollar position in the risky asset. Self-financing and discounting give $dX_t = u_t(\,(\mu - r)\,dt + \sigma\,dW_t\,) = \sigma u_t(\rho\,dt + dW_t)$. Substituting the innovation decomposition $dW_t = d\widehat{W}_t - (\rho - m_t)\,dt$:
\begin{equation}
\label{eq:wealth}
    dX_t = \sigma u_t (m_t\,dt + d\widehat{W}_t).
\end{equation}

The drift $\sigma m_t u_t$ is bilinear, a product of belief and control. This breaks from classical LQG, where the drift is jointly affine in state and control. The bilinear structure prevents the HJB equation from closing into Riccati ODEs (Proposition~\ref{prop:polynomial}) and requires the exponential substitution of Section~\ref{sec:closed-form}.

Under the observation filtration, the pair $(X_t, m_t)$ is Markov, with $P_t$ as a deterministic time-dependent parameter. Through filtering, we have gone from partial information to a fully observed Markov control problem.

\subsection{Entropy-Regularized Objective}

Following \citet{wang2019}, we replace the pointwise control $u_t$ with a stochastic policy $\pi_t$, a probability distribution over positions.

\begin{definition}[Admissible policy]
\label{def:admissible}
An admissible policy is a family $\{\pi_t\}_{t \in [0,T]}$ of probability densities on $\R$, progressively measurable with respect to the observation filtration $(\F_t)_{t \in [0,T]}$, satisfying $\E\!\left[\int_0^T \int_\R u^2 \pi_t(u)\,du\,dt\right] < \infty$.
\end{definition}

We work at the generator level, where the controlled dynamics depend on $\pi$ only through its moments; this is the relaxed (exploratory) control formulation developed in \citet{wang2019b} and applied to mean--variance in \citet{wang2019}.

The entropy-regularized value function is
\begin{equation}
\label{eq:value}
    V(t,x,m,P) = \inf_{\pi \text{ admissible}} \E\!\left[ (X_T - w)^2 + \tau \int_t^T H(\pi_s)\,ds \;\Big|\; X_t = x, m_t = m, P_t = P \right],
\end{equation}
where $w \in \R$ is the target wealth level, $\tau > 0$ is the entropy weight, and $H(\pi) := \int \pi(u) \log \pi(u)\,du$ is the negative of differential entropy. We treat the posterior variance $P$ as a free initial condition for the dynamic programming argument; since $P_t$ is in fact deterministic, the value along the realized trajectory is recovered by restricting to the characteristic $\{P = P_t\}$ in Section~\ref{sec:closed-form}. Following the classical embedding of \citet{zhou2000}, $w$ plays the role of a Lagrange multiplier for the mean constraint $\mathbb{E}[X_T]=z$, and sweeping $w$ traces the efficient frontier. 

The entropy term penalizes concentrated policies. As $\tau \to 0$, we recover the deterministic Bayesian Markowitz problem.

\section{HJB Equation and Optimal Policy}
\label{sec:hjb}

\subsection{Generator}

Under policy $\pi_t$ with mean $\bar{u}$ and variance $\varsigma^2$, the generator of any smooth $\varphi(t,x,m,P)$ is
\begin{equation}
\label{eq:generator}
    \mathcal{L}^\pi \varphi = \sigma \bar{u} m\,\varphi_x + \frac{\sigma^2}{2}(\bar{u}^2 + \varsigma^2)\varphi_{xx} + \sigma \bar{u} P\,\varphi_{xm} + \frac{P^2}{2}\varphi_{mm} - P^2\varphi_P.
\end{equation}

The policy enters only through its first two moments: $\E_\pi[u] = \bar{u}$ and $\E_\pi[u^2] = \bar{u}^2 + \varsigma^2$. The mixed term $\varphi_{xm}$ involves $\bar{u}$ only because $dX\,dm = \sigma u P_t\,dt$ is linear in $u$.

\subsection{HJB Equation}

Since $(X_t, m_t, P_t)$ is Markov under the observation filtration (with $P_t$ deterministic), the dynamic programming principle applies. The HJB equation is
\begin{equation}
\label{eq:hjb}
    \partial_t V + \inf_{\pi} \left\{ \mathcal{L}^\pi V + \tau H(\pi) \right\} = 0, \qquad V(T,x,m,P) = (x-w)^2.
\end{equation}

\subsection{Gaussian Optimality}

\begin{lemma}[Optimal policy is Gaussian]
\label{lem:gaussian}
Suppose $V$ is $C^1$ in $(t,P)$ and $C^2$ in $(x,m)$, with $V_{xx} > 0$. The optimal policy is $\pi^* = \N(\bar{u}^*, \varsigma^{*2})$ with
\begin{equation}
\label{eq:optimal-policy}
    \bar{u}^* = -\frac{mV_x + PV_{xm}}{\sigma V_{xx}}, \qquad \varsigma^{*2} = \frac{\tau}{\sigma^2 V_{xx}}.
\end{equation}
\end{lemma}

\begin{proof}
The $\pi$-dependent terms in $\mathcal{L}^\pi V + \tau H(\pi)$ are
\[
    \Phi(\pi) := \sigma \bar{u} (mV_x + PV_{xm}) + \frac{\sigma^2}{2}(\bar{u}^2 + \varsigma^2)V_{xx} + \tau H(\pi).
\]
Taking the first variation under $\int \pi = 1$ gives
\[
    \log \pi^*(u) + 1 + \frac{1}{\tau}\left( \sigma u (mV_x + PV_{xm}) + \frac{\sigma^2 u^2}{2}V_{xx} \right) = \text{const}.
\]
The exponent is quadratic in $u$, so $\pi^*$ is Gaussian. Since $\Phi$ is linear in $\pi$ apart from the strictly convex entropy term $\tau H(\pi)$, this stationary point is the unique minimizer; completing the square gives the stated mean and variance. 
\end{proof}

\begin{remark}[Mean--variance separation]
\label{rem:mv-separation}
The optimal policy separates cleanly: the mean $\bar{u}^*$ depends on beliefs and value gradients, while the variance $\varsigma^{*2}$ depends only on the curvature $V_{xx}$ and the entropy weight $\tau$. The coupling term $PV_{xm}$ in the mean is new relative to the known-drift case and reflects drift uncertainty. The mean control contains no explicit $\tau$; Section~\ref{sec:value} shows the implicit dependence through $V$ vanishes as well, since all $\tau$-dependence of $V$ is confined to a wealth-independent additive term. Entropy regularization therefore affects only the policy variance.
\end{remark}

\subsection{Reduced HJB}

Substituting the Gaussian optimizer into \eqref{eq:hjb}:
\begin{equation}
\label{eq:hjb-reduced}
    0 = V_t + \frac{P^2}{2}V_{mm} - P^2 V_P - \frac{(mV_x + PV_{xm})^2}{2V_{xx}} - \frac{\tau}{2}\log\!\left( \frac{2\pi\tau}{\sigma^2 V_{xx}} \right),
\end{equation}
with terminal condition $V(T,x,m,P) = (x-w)^2$.

\section{Value Function Structure}
\label{sec:value}

In classical LQG and in \citet{wang2019}, the value function is polynomial
in the state, with coefficients that solve ODEs in time. The bilinear drift 
$\sigma m_t u_t$ prevents this here; the nonlinear term 
$\frac{(mV_x + PV_{xm})^2}{V_{xx}}$ increases the degree in $m$ faster 
than the linear PDE terms can cancel.

\begin{proposition}[Polynomial impossibility]
\label{prop:polynomial}
For $P > 0$, no polynomial ansatz in $(x,m)$ of finite degree with $V_{xx} \not\equiv 0$ satisfies the reduced HJB \eqref{eq:hjb-reduced}. (See Appendix~\ref{app:polynomial} for the full proof.)
\end{proposition}

So we look for a solution quadratic in wealth but nonpolynomial in belief.

\subsection{Quadratic-in-Wealth Ansatz}

The terminal condition is quadratic in just $x$. So we seek a solution of the form
\begin{equation}
\label{eq:ansatz}
    V(t,x,m,P) = A(t,m,P)x^2 + B(t,m,P)x + C(t,m,P),
\end{equation}
where the coefficients $A$, $B$, $C$ depend on beliefs $(m, P)$ but need not be polynomial in $m$. We verify below that this ansatz satisfies the HJB.

\begin{proposition}[Coefficient system]
\label{prop:coefficient-system}
The ansatz \eqref{eq:ansatz} with $A > 0$ satisfies \eqref{eq:hjb-reduced} if and only if $(A,B,C)$ solve
\begin{align}
    0 &= A_t + \frac{P^2}{2}A_{mm} - P^2 A_P - \frac{\Gamma^2}{A}, \label{eq:A-pde} \\
    0 &= B_t + \frac{P^2}{2}B_{mm} - P^2 B_P - \frac{\Gamma\Lambda}{A}, \label{eq:B-pde} \\
    0 &= C_t + \frac{P^2}{2}C_{mm} - P^2 C_P - \frac{\Lambda^2}{4A} - \frac{\tau}{2}\log\!\left(\frac{\pi\tau}{\sigma^2 A}\right), \label{eq:C-pde}
\end{align}
where $\Gamma := mA + PA_m$ and $\Lambda := mB + PB_m$, with terminal conditions $A(T) = 1$, $B(T) = -2w$, $C(T) = w^2$.
\end{proposition}

The system is triangular: \eqref{eq:A-pde} is closed in $A$; given $A$, \eqref{eq:B-pde} is linear in $B$; given $(A,B)$, \eqref{eq:C-pde} is linear in $C$.

\begin{proposition}[$w$-separation]
\label{prop:w-separation}
Within the system \eqref{eq:A-pde}--\eqref{eq:C-pde}: (i) $A$ is independent of $w$; (ii) $B = -2wA$; (iii) writing $C = w^2 A + D$, the function $D$ is independent of $w$ and satisfies
\begin{equation}
\label{eq:D-pde}
    0 = D_t + \frac{P^2}{2}D_{mm} - P^2 D_P - \frac{\tau}{2}\log\!\left(\frac{\pi\tau}{\sigma^2 A}\right), \qquad D(T) = 0.
\end{equation}
\end{proposition}

So the problem reduces to solving \eqref{eq:A-pde} for $A$, then \eqref{eq:D-pde} for $D$.

\subsection{Closed-form Solution}
\label{sec:closed-form}

Since $P_t = P_0/(1+P_0 t)$ is deterministic, $P$ in the HJB \eqref{eq:hjb-reduced} can be eliminated by restricting to the characteristic $\{P = P_t\}$. Define
\[
\widetilde{V}(t,x,m) := V(t,x,m,P_t), \qquad \widetilde{A}(t,m) := A(t,m,P_t).
\]
Differentiating along the characteristic using $dP_t/dt = -P_t^2$ gives
\[
\widetilde{V}_t \;=\; V_t - P_t^2\, V_P,
\]
so the $-P^2 V_P$ term is absorbed into the time derivative and the reduced PDE for $\widetilde{V}$ carries no explicit $P$-dependence.

We then write $A(t,m)$ for $\widetilde{A}(t,m)$ when no ambiguity arises.

The exponential substitution $A = e^R$---the analogue of the Bayesian risk-premium representation of \citet[Theorem 4.1]{defranco2018}---transforms \eqref{eq:A-pde} into a semilinear PDE; as in \citet[Section 4.3.2]{defranco2018}, it admits a quadratic solution $R(t,m) = \alpha(t)m^2 + \gamma(t)$.

\begin{proposition}[Closed-form solution]
\label{prop:closed-form}
For Gaussian priors, $A(t,m) = e^{\alpha(t)m^2 + \gamma(t)}$ with
\begin{equation}
\label{eq:alpha}
    \alpha(t) = -\frac{(1 + P_0 t)(T - t)}{1 + P_0(2T - t)},
\end{equation}
and
\begin{equation}
\label{eq:gamma}
    \gamma(t) = \int_t^T P_s^2 \alpha(s)\,ds = \frac{1}{2}\log\!\left(\frac{(1 + P_0 t)(1 + P_0(2T - t))}{(1 + P_0 T)^2}\right).
\end{equation}
\end{proposition}

\begin{proof}[Proof sketch]
A linear term $\beta(t)m$ in $R$ is consistent only with $\beta \equiv 0$ (it satisfies a linear homogeneous ODE with $\beta(T) = 0$), reflecting evenness in $m$. With $A = e^{\alpha m^2 + \gamma}$ and $P = P_t$ deterministic, the $A$-PDE reduces to matching coefficients:
\begin{align*}
    m^2\text{-terms:} &\quad \alpha' = 1 + 4P_t\alpha + 2P_t^2\alpha^2, \quad \alpha(T) = 0, \\
    m^0\text{-terms:} &\quad \gamma' = -P_t^2\alpha, \quad \gamma(T) = 0.
\end{align*}
The $\alpha$-equation is a Riccati ODE solved by the stated formula. Integrating $\gamma' = -P_t^2\alpha$ from $t$ to $T$ with terminal condition $\gamma(T) = 0$ gives $\gamma(t) = \int_t^T P_s^2 \alpha(s)\,ds$. The closed form follows by direct integration.
\end{proof}

\begin{remark}[$\tau$ independence of $A$]
\label{rem:tau-absent}
The entropy weight $\tau$ does not appear in $A$ or $\alpha$; it enters only the $D$ equation. The coefficient $A$ determines $V_{xx} = 2A$, which is set by the underlying Bayesian Markowitz problem regardless of whether the agent randomizes. Since $\alpha(t) < 0$ and $\gamma(t) < 0$ for $t < T$, we have $A(t,m) < 1$ for all $m$ whenever $t < T$: drift uncertainty reduces the value-function curvature relative to the terminal condition. (At $m = 0$ this is $A(t,0) = e^{\gamma(t)} < 1$; $\gamma(t) < 0$ follows from AM--GM applied to $(1 + P_0 t)$ and $(1 + P_0(2T-t))$, whose sum $2(1 + P_0 T)$ is constant.)

The formula for $\alpha$ is more readable in terms of current posterior variance:
\[
    \alpha(t) = -\frac{T - t}{1 + 2P_t(T - t)}.
\]
\end{remark}

\subsection{The Entropy Premium}

Proposition~\ref{prop:w-separation} reduced the problem to solving \eqref{eq:D-pde} for $D$. Since $P_t$ is deterministic, we work along the characteristic and write $\tilde{D}(t,m) := D(t,m,P_t)$. Using $\frac{dP_t}{dt} = -P_t^2$, the PDE becomes
\begin{equation}
\label{eq:D-reduced}
    0 = \tilde{D}_t + \frac{P_t^2}{2}\tilde{D}_{mm} - \frac{\tau}{2}\log\!\left(\frac{\pi\tau}{\sigma^2 A(t,m)}\right), \qquad \tilde{D}(T,m) = 0.
\end{equation}

Substituting $A = e^{\alpha m^2 + \gamma}$, the source term becomes
\[
    \log\!\left(\frac{\pi\tau}{\sigma^2 A}\right) = \log\frac{\pi\tau}{\sigma^2} - \alpha(t)m^2 - \gamma(t),
\]
which is quadratic in $m$. This motivates the ansatz $\tilde{D}(t,m) = \eta(t)m^2 + \zeta(t)$.

\begin{proposition}[Entropy premium]
\label{prop:entropy-premium}
The function $D$ is given by $D(t,m,P_t) = \eta(t)m^2 + \zeta(t)$, where
\begin{align}
    \eta(t) &= \frac{\tau}{2}\int_t^T \alpha(s)\,ds, \label{eq:eta} \\
    \zeta(t) &= -\int_t^T \left( \frac{\tau}{2}\left[\log\frac{\pi\tau}{\sigma^2} - \gamma(s)\right] - P_s^2\,\eta(s) \right) ds. \label{eq:zeta}
\end{align}
Since $\alpha(t) < 0$ for $t < T$, we have $\eta(t) < 0$ for $t < T$.
\end{proposition}

\begin{proof}
Substituting $\tilde{D} = \eta m^2 + \zeta$ into \eqref{eq:D-reduced} and matching coefficients:
\begin{align*}
    m^2\text{-terms:} &\quad \eta'(t) = -\frac{\tau}{2}\alpha(t), \quad \eta(T) = 0, \\
    m^0\text{-terms:} &\quad \zeta'(t) = \frac{\tau}{2}\left[\log\frac{\pi\tau}{\sigma^2} - \gamma(t)\right] - P_t^2\,\eta(t), \quad \zeta(T) = 0.
\end{align*}
Integrating backward from $T$ gives \eqref{eq:eta}--\eqref{eq:zeta}.
\end{proof}

\begin{remark}[Probabilistic form of the premium]
A Feynman--Kac representation of \eqref{eq:D-reduced} gives
\[
    D(t,m) = \tau\,\E\!\left[\int_t^T H(\pi^*_s)\,ds \;\Big|\; m_t = m\right] + \frac{\tau}{2}(T - t),
\]
the expected entropy cost plus a deterministic noise-injection cost: along $\pi^*$, $\tfrac{1}{2}\sigma^2\varsigma^{*2}V_{xx} \equiv \tfrac{\tau}{2}$, so each unit of time contributes exactly $\tau/2$. This is the value-function counterpart of the policy-level fact that randomization is purchased at a fixed marginal rate.
\end{remark}

The complete value function is thus
\begin{equation}
\label{eq:value-complete}
    V(t,x,m) = e^{\alpha(t)m^2 + \gamma(t)}(x-w)^2 + \eta(t)m^2 + \zeta(t).
\end{equation}

\begin{remark}[Policy independence of entropy premium]
\label{rem:policy-independence}
The entropy premium $D = \eta m^2 + \zeta$ is independent of wealth $x$. Since the optimal policy \eqref{eq:optimal-policy} depends only on $V_x$, $V_{xm}$, and $V_{xx}$, all of which involve only $A$, the function $D$ does not affect optimal behavior. It captures the expected cost of entropy regularization; it does not guide the optimal policy.
\end{remark}

\subsection{Optimal Policy in Closed Form}

Using Lemma~\ref{lem:gaussian} with $V_{xx} = 2A$ and $B = -2wA$:
\begin{equation}
\label{eq:policy-closed}
    \bar{u}^* = -\frac{m(1 + 2P_t\alpha)}{\sigma}(x - w), \qquad \varsigma^{*2} = \frac{\tau}{2\sigma^2 A}.
\end{equation}

The mean control is linear in tracking error $(x - w)$ with belief-dependent gain. The policy variance depends on $(t,m)$ through $A(t,m)$, but is independent of wealth $x$ and target $w$.

\begin{remark}[Verification]
\label{rem:verification}
The candidate value function \eqref{eq:value-complete} is a classical solution to the HJB equation \eqref{eq:hjb-reduced}. Verification follows the finite-horizon argument of \citet[Theorem 3.5.2]{pham2009}, whose hypotheses we adapt to the present problem. Define the cost-adjusted process
\[
    M_t^\pi := V(t, X_t, m_t) + \tau \int_0^t H(\pi_s)\,ds.
\]
Applying It\^{o}'s formula, the HJB inequality $V_t + \mathcal{L}^\pi V + \tau H(\pi) \geq 0$ (with equality for $\pi^*$) shows that $M_t^\pi$ is a local submartingale under any admissible $\pi$ and a local martingale under $\pi^*$. The candidate \eqref{eq:value-complete} has quadratic growth in $(x,m)$---since $\alpha(t) < 0$, the factor $e^{\alpha m^2}$ is bounded, so $|V| \leq e^{\gamma(t)}(x-w)^2 + |\eta(t)|m^2 + |\zeta(t)|$---which together with the moment bounds of Proposition~\ref{prop:admissibility} lets us localize and pass to the limit by uniform integrability. Taking expectations gives
\[
    V(0,x_0,m_0) \leq \E\!\left[(X_T - w)^2 + \tau \int_0^T H(\pi_s)\,ds\right]
\]
for all admissible $\pi$, with equality under the optimal policy $\pi^*$ when it is itself admissible---that is, when $P_0 T < 7/2$ (Proposition~\ref{prop:admissibility})---confirming that $V$ equals the value function \eqref{eq:value}.
\end{remark}

\begin{proposition}[Admissibility of optimal policy]
\label{prop:admissibility}
The optimal Gaussian policy \eqref{eq:policy-closed} is admissible in the sense of Definition~\ref{def:admissible} if and only if
\begin{equation}
\label{eq:admissibility-condition}
    \sup_{t \in [0,T]}\; \frac{2P_0^2\,t(T-t)}{1 + P_0(2T - t)} < 1,
\end{equation}
equivalently the clean parameter restriction $P_0 T < 7/2$. The condition controls
\[
    \E\!\left[e^{|\alpha(t)|\,m_t^2}\right] < \infty \quad \text{for all } t \in [0,T],
\]
which is what makes both the variance and the mean components of $\E_{\pi^*}[u^2]$ integrable. All numerical examples in this paper satisfy \eqref{eq:admissibility-condition}; for instance, with $P_0 = 1$ and $T = 1$, the supremum equals $10 - 4\sqrt{6} \approx 0.20$.
\end{proposition}

\begin{proof}
Admissibility requires $\E[\int_0^T \int_\R u^2\,\pi^*_t(u)\,du\,dt] < \infty$, i.e.\ both
\[
    \E\!\left[\int_0^T \varsigma^{*2}(t,m_t)\,dt\right] < \infty
    \qquad \text{and} \qquad
    \E\!\left[\int_0^T \bar{u}^{*2}(t,X_t,m_t)\,dt\right] < \infty.
\]

\emph{Variance.} Since $\alpha(t) < 0$,
\[
    \varsigma^{*2} = \frac{\tau}{2\sigma^2}\,e^{|\alpha(t)|\,m_t^2 - \gamma(t)}
\]
grows exponentially in $m_t^2$, so $\E[\varsigma^{*2}] < \infty$ iff $\E[e^{|\alpha(t)|\,m_t^2}] < \infty$. The posterior mean is $m_t = m_0 + \int_0^t P_s\,d\widehat{W}_s$, so by It\^o isometry
\[
    \Var(m_t) = \int_0^t P_s^2\,ds = \frac{P_0^2\,t}{1 + P_0\,t}.
\]
For $Z \sim \N(\mu,v)$ one has $\E[e^{cZ^2}] < \infty$ iff $2cv < 1$; with $c = |\alpha(t)|$ and $v = \Var(m_t)$,
\[
    2|\alpha(t)|\,\Var(m_t) = \frac{2P_0^2\,t(T-t)}{1 + P_0(2T - t)},
\]
so the pointwise moment is finite iff this is $< 1$, which is \eqref{eq:admissibility-condition}. Writing $z := P_0 T$, $s := t/T$,
\[
    \sup_{s \in [0,1]} \frac{2 z^2 s(1-s)}{1 + z(2-s)} < 1,
\]
and a direct calculation (interior maximizer $s_* = 4/7$ at $z = 7/2$) shows the supremum equals $1$ exactly at $z = 7/2$, giving $P_0 T < 7/2$. Since the left side of \eqref{eq:admissibility-condition} is continuous in $t$ on $[0,T]$, the condition makes $\inf_t\,(1 - 2|\alpha(t)|\Var(m_t)) > 0$, so $\E[e^{|\alpha(t)|m_t^2}]$, hence $\E[\varsigma^{*2}(t,m_t)]$, is bounded uniformly in $t$, and $\int_0^T \E[\varsigma^{*2}]\,dt < \infty$.

\emph{Mean.} Write $Z_t := X_t - w$ and $h(t) := 1 + 2P_t\alpha(t) = (1 + 2P_t(T-t))^{-1} \in (0,1]$, so $\bar{u}^* = -m_t h(t) Z_t/\sigma$. It\^o's formula under $\pi^*$ gives
\[
    dZ_t^2 = \big[\,m_t^2 h(t)(h(t)-2)Z_t^2 + \sigma^2\varsigma^{*2}(t,m_t)\,\big]\,dt + dN_t,
\]
with $N_t$ a local martingale. As $h \in (0,1]$, the coefficient $h(h-2) < 0$: tracking error mean-reverts. Localizing at $T_n := \inf\{t : |Z_t| + |m_t| > n\}$, taking expectations, and applying Fatou to the nonnegative term,
\[
    \E\!\int_0^T m_t^2\,h(t)(2 - h(t))Z_t^2\,dt \;\le\; Z_0^2 + \sigma^2\!\int_0^T \E[\varsigma^{*2}(t,m_t)]\,dt \;<\; \infty.
\]
Since $h(t)(2-h(t)) \ge h(t) \ge (1 + 2P_0 T)^{-1} > 0$, this gives $\E[\int_0^T m_t^2 Z_t^2\,dt] < \infty$, the mean-control bound.

\emph{Necessity.} If the supremum in \eqref{eq:admissibility-condition} is $\ge 1$, the pointwise moment criterion fails badly enough that $\int_0^T \E[\varsigma^{*2}]\,dt = \infty$: for $P_0 T > 7/2$ it fails on an open interval of times, while at $P_0 T = 7/2$ it fails at the single instant $t_* = 4T/7$, where the quadratic touching gives $\E[\varsigma^{*2}(t,m_t)] \sim |t - t_*|^{-1}$. Hence $\pi^*$ is inadmissible.
\end{proof}

\begin{assumption}
\label{assump:admissible}
Throughout the remainder of the paper, $P_0 T < 7/2$, so that the optimal policy is admissible (Proposition~\ref{prop:admissibility}).
\end{assumption}

\section{Limiting Cases}
\label{sec:limits}

Two limits recover existing results.

\subsection{Known Drift: $P_0 \to 0$}

When $P_0 \to 0$, the posterior collapses: $P_t \equiv 0$ and $m_t \equiv m_0 =: \rho$. From \eqref{eq:alpha}, $\alpha(t) \to -(T - t)$. From \eqref{eq:gamma}, $\gamma(t) \to 0$ since $P_s \to 0$. Thus $A(t) \to e^{-(T-t)\rho^2}$.

The reduced HJB becomes
\[
    0 = V_t - \frac{\rho^2 V_x^2}{2V_{xx}} - \frac{\tau}{2}\log\!\left( \frac{2\pi\tau}{\sigma^2 V_{xx}} \right),
\]
and the optimal policy is
\[
    \bar{u}^* = -\frac{\rho}{\sigma}(x - w), \qquad \varsigma^{*2} = \frac{\tau}{2\sigma^2}e^{\rho^2(T-t)},
\]
recovering \citet{wang2019}, Theorem 1.

\subsection{Deterministic Control: $\tau \to 0$}

When $\tau \to 0$, the policy variance $\varsigma^{*2} = \frac{\tau}{2\sigma^2 A} \to 0$. The policy collapses to deterministic feedback:
\[
    u^* = \bar{u}^* = -\frac{m(1 + 2P_t\alpha)}{\sigma}(x - w),
\]
matching \citet{defranco2018}.

\section{Discussion}
\label{sec:discussion}

Sections~\ref{sec:hjb}--\ref{sec:value} established the closed-form solution; we now read off what it means.

\subsection{Orthogonality of Entropy and Learning}

The posterior dynamics \eqref{eq:dm}--\eqref{eq:dP} depend on the observation process alone: the control $\pi$ does not appear. The investor observes prices, not portfolio returns, so randomizing the policy cannot accelerate learning about $\rho$. The mean control is $\tau$-independent, the curvature $V_{xx} = 2A$ is $\tau$-independent, and entropy enters only through the policy variance $\varsigma^{*2} = \tau/(2\sigma^2 A)$ and the additive premium $D$ (which does not affect the optimal policy).

\subsection{Belief-Dependent Policy Variance}

Recall from \eqref{eq:policy-closed} that $\varsigma^{*2}(t,m) = \frac{\tau}{2\sigma^2 A(t,m)}$ with $A = e^{\alpha(t)m^2 + \gamma(t)}$.

\begin{proposition}[Belief-dependent variance]
\label{prop:beliefvar}
Under Assumption~\ref{assump:prior} with $\tau > 0$, the optimal policy variance $\varsigma^{*2}(t,m)$ is: symmetric in $m$, minimized at $m = 0$, and strictly increasing in $|m|$ for $|m| > 0$ at each fixed $t < T$.
\end{proposition}

\begin{proof}
From \eqref{eq:policy-closed},
\[
    \varsigma^{*2}(t,m) = \frac{\tau}{2\sigma^2}\,e^{-\alpha(t)m^2 - \gamma(t)}.
\]
This depends on $m$ only through $m^2$, giving symmetry. Differentiating in $m$:
\[
    \frac{\partial\,\varsigma^{*2}}{\partial m} = -2\alpha(t)\,m\,\varsigma^{*2}(t,m).
\]
By Proposition~\ref{prop:closed-form}, $\alpha(t) < 0$ for $t < T$, so $-2\alpha(t) > 0$. The derivative has the same sign as $m$; negative for $m < 0$, zero at $m = 0$, positive for $m > 0$. Hence $\varsigma^{*2}$ is minimized at $m = 0$ and strictly increasing in $|m|$ for $|m| > 0$.
\end{proof}

The mechanism runs through the belief, not the control: large $|m|$ simultaneously makes the mean control aggressive ($|\bar{u}^*| \propto |m|$) and, since $\alpha(t) < 0$, shrinks the curvature coefficient $A(t,m) = e^{\alpha m^2 + \gamma}$, so that $\varsigma^{*2} = \tau/(2\sigma^2 A)$ grows. The variance responds to belief magnitude directly, not to the size of the position.
This is already present in the known-drift model, where $\varsigma^{*2} = (\tau/2\sigma^2)\exp(\rho^2(T-t))$ grows with $|\rho|$. Partial information replaces $\rho$ with $m_t$ and $(T-t)$ with $|\alpha(t)|$, modifying the coefficients but not the structure.

\subsection{Entropy Regularization Is Not Exploration}

In reinforcement learning with an unknown environment, entropy regularization pushes the agent toward diverse actions, and the resulting data sharpens its estimates; policy variance shrinks as the agent learns. None of that applies here. The observation channel is exogenous to the control, so randomizing the position reveals nothing about $\rho$, and the entropy penalty does no more than spread probability mass around the mean control. That spread is widest when the mean control is most aggressive (large $|m_t|$). Along sample paths its only stochastic driver is $m_t$; the posterior variance $P_t$ enters $\varsigma^{*2}$ only deterministically, through the coefficient $\alpha(t)$, where it acts to dampen rather than amplify the dependence on $m_t$ (larger $P_t$ lowers $|\alpha(t)|$).

\subsection{On Robustness}

One might interpret the belief-dependent variance as robustness. The agent randomizes more when its position is most aggressive, hedging against the possibility that the posterior is wrong. However, the added noise is zero-mean and symmetric: it can amplify losses as easily as mitigate them, and it is purchased entirely by the entropy bonus (the premium $D$ of Proposition~\ref{prop:entropy-premium}). A formal connection to distributionally robust control---via Donsker--Varadhan duality, for example---would require an adversarial structure that is not present here.

\subsection{Numerical Illustration}

Figure~\ref{fig:variance} visualizes Proposition~\ref{prop:beliefvar}: the policy variance is symmetric in $m$, minimized at $m = 0$, and grows exponentially in $|m|$, with the dependence flattening as $t \to T$ and $\alpha(t) \to 0$.

\begin{figure}[htbp]
    \centering
    \includegraphics[width=\textwidth]{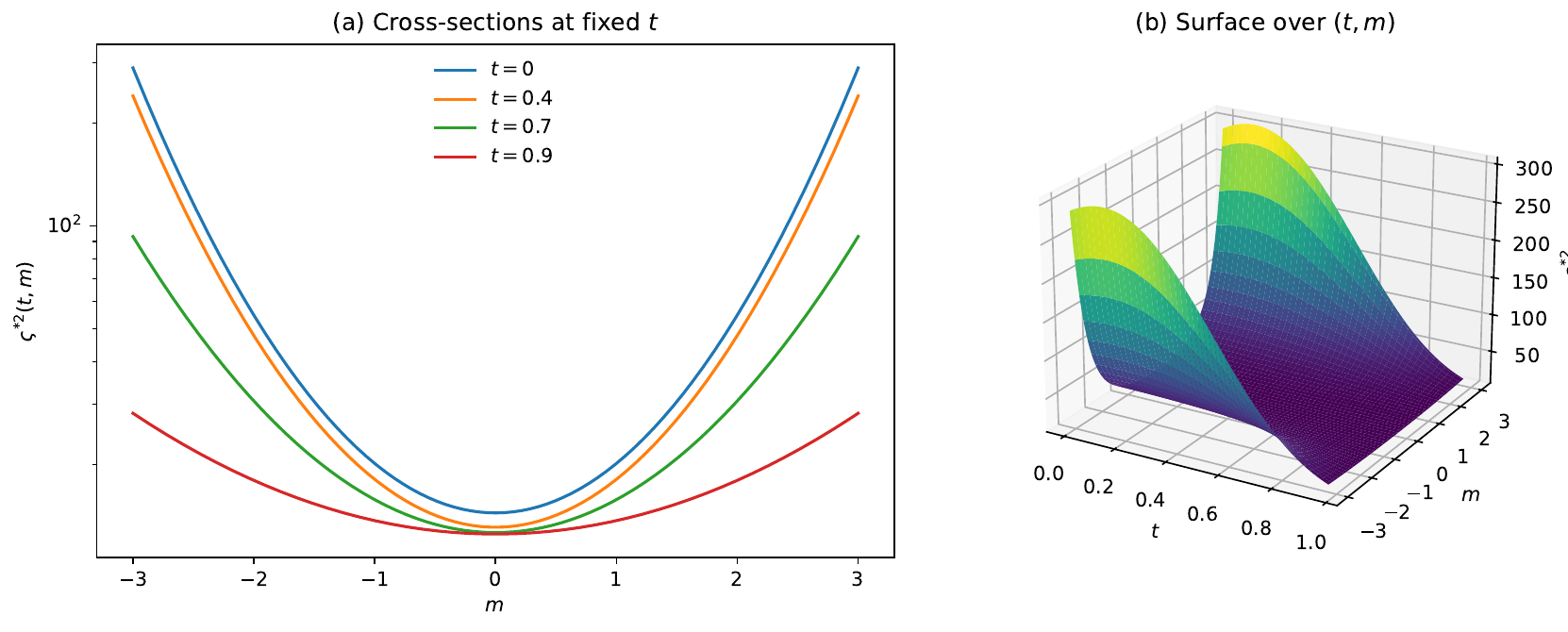}
    \caption{Optimal policy variance $\varsigma^{*2}(t,m)$. (a) Cross-sections at fixed $t$: the variance is symmetric in $m$, minimized at $m=0$, and exponentially increasing in $|m|$ at each $t < T$; the curve flattens as $t \to T$ since $\alpha(T) = 0$. (b) Surface over the $(t,m)$ plane: variance is largest where $|m|$ is large and time-to-go is long. Parameters: $P_0 = 1$, $T=1$, $\tau=1$, $\sigma=0.2$.}
    \label{fig:variance}
\end{figure}

\section{Conclusion}
\label{sec:conclusion}

We solved the entropy-regularized mean--variance problem under Bayesian drift uncertainty, filling the remaining cell of the $2 \times 2$ table (known vs.\ unknown drift $\times$ deterministic vs.\ entropy-regularized control), whose other three cells are classical mean--variance, \citet{defranco2018}, and \citet{wang2019}. The closed-form solution is clean and both parent models are limiting cases. The main finding is that the two mechanisms are orthogonal: when observations are exogenous to the control, entropy cannot explore, and the belief-dependent policy variance it produces is the same phenomenon as in the known-drift case with $m_t$ in place of $\rho$.

\begin{remark}[Beyond orthogonality]
\label{rem:extensions}
Orthogonality is a property of the Shannon penalty, which enters the optimal policy through its variance alone and leaves the mean control free of $\tau$. A natural direction for future work is a regularizer that does not separate this way: one whose weight also shapes the mean control, so that regularization and the belief-driven trade interact rather than run in parallel.
\end{remark}
\newpage
\section*{Acknowledgments}
The author thanks Prof.\ Mark Kon for advising this project and Prof.\ Konstantinos Spiliopoulos for helpful discussions.

\bibliographystyle{plainnat}
\bibliography{references}

\newpage
\appendix

\section{Proof of Polynomial Impossibility}
\label{app:polynomial}

\begin{proof}[Proof of Proposition~\ref{prop:polynomial}]
Fix $t$ and $P > 0$, and suppose $V(t,\cdot,\cdot,P)$ is polynomial of finite degree in $(x,m)$ with $V_{xx} \not\equiv 0$, so $p := \deg_x V \geq 2$. Write $V = \sum_{k=0}^p a_k(t,m,P)x^k$ with $a_p \not\equiv 0$. All degree counts below are in $m$ at fixed $(t,P)$; the $t$- and $P$-derivatives preserve or lower $m$-degree.

Define $G := mV_x + PV_{xm}$. The leading term is $G = p(ma_p + Pa_{p,m})x^{p-1} + O(x^{p-2})$. Thus
\[
    \frac{G^2}{2V_{xx}} = \frac{p}{2(p-1)}\frac{(ma_p + Pa_{p,m})^2}{a_p}x^p + O(x^{p-1}).
\]

Dividing \eqref{eq:hjb-reduced} by $x^p$ and taking $x \to \infty$:
\[
    0 = \partial_t a_p + \frac{P^2}{2}\partial_{mm}a_p - P^2\partial_P a_p - \frac{p}{2(p-1)}\frac{(ma_p + Pa_{p,m})^2}{a_p}.
\]

The bracket $ma_p + Pa_{p,m}$ cannot vanish identically in $m$: the equation $ma + Pa_m = 0$ forces $a \propto e^{-m^2/(2P)}$ (using $P > 0$), which is polynomial only if $a \equiv 0$, contradicting $a_p \not\equiv 0$. The nonlinear term is therefore genuinely present.

Let $r = \deg_m a_p$. The final term grows like $m^{r+2}$ as $|m| \to \infty$, while linear terms grow at most like $m^r$. Contradiction. 
\end{proof}

\section{Notation}
\label{app:notation}

\begin{center}
\renewcommand{\arraystretch}{1.35}
\begin{tabular}{p{1.8cm} p{11cm}}
\toprule
\multicolumn{2}{l}{\textbf{Market}} \\
\midrule
$S_t$ & Risky asset price \\
$\mu, \sigma, r$ & Drift (unknown), volatility (known), risk-free rate \\
$\rho$ & Sharpe ratio $(\mu - r)/\sigma$ (unknown) \\
$W_t$ & Brownian motion driving prices \\
$T$ & Investment horizon \\
\midrule
\multicolumn{2}{l}{\textbf{Filtering}} \\
\midrule
$m_t$ & Posterior mean of $\rho$ given observations up to time $t$ \\
$P_t$ & Posterior variance of $\rho$ given observations up to time $t$ \\
$m_0, P_0$ & Prior mean and variance of $\rho$ \\
$\widehat{W}_t$ & Innovation process (Brownian motion adapted to observations) \\
\midrule
\multicolumn{2}{l}{\textbf{Control}} \\
\midrule
$X_t$ & Discounted wealth \\
$u_t$ & Position in risky asset (discounted dollars) \\
$\pi_t$ & Policy: probability distribution over positions at time $t$ \\
$\bar{u}, \varsigma^2$ & Mean and variance of policy $\pi$ \\
$w$ & Target wealth level in the objective; Lagrange multiplier for the $\E[X_T] = z$ constraint \\
$z$ & Target expected terminal wealth, $\E[X_T] = z$ (embedding constraint) \\
$\tau$ & Entropy weight (regularization parameter) \\
$H(\pi)$ & Negative of differential entropy: $\int \pi(u)\log\pi(u)\,du$ \\
\midrule
\multicolumn{2}{l}{\textbf{Value Function}} \\
\midrule
$V$ & Value function $V(t,x,m,P)$ \\
$A, B, C$ & Coefficients in ansatz $V = Ax^2 + Bx + C$ \\
$\alpha, \gamma$ & Functions of $t$ in closed-form $A = e^{\alpha m^2 + \gamma}$ \\
$\eta, \zeta$ & Functions of $t$ in entropy premium $D = \eta m^2 + \zeta$ \\
\bottomrule
\end{tabular}
\end{center}

\end{document}